\documentclass[12pt]{amsart}
\usepackage{amsfonts,amssymb,amscd,amsmath,enumerate,color, xypic}
\def\NZQ{\mathbb}               

\def\QQ{{\NZQ Q}}
\def\ZZ{{\NZQ Z}}
\def\RR{{\NZQ R}}

\def\frk{\mathfrak}               

\def\Phi{{\frk N}}
\def\eb{{\bold e}}

\def\opn#1#2{\def#1{\operatorname{#2}}} 
\opn\chara{char} 
\opn\length{\ell} 
\opn\pd{pd} 
\opn\rk{rk}
\opn\projdim{proj\,dim} 
\opn\injdim{inj\,dim} 
\opn\rank{rank}
\opn\depth{depth} 
\opn\grade{grade} 
\opn\height{height}
\opn\embdim{emb\,dim} 
\opn\codim{codim}

\opn\Tr{Tr} 
\opn\bigrank{big\,rank}
\opn\superheight{superheight}
\opn\lcm{lcm}
\opn\trdeg{tr\,deg}
\opn\reg{reg} 
\opn\lreg{lreg} 
\opn\ini{in} 
\opn\lpd{lpd}
\opn\size{size}
\opn\mult{mult}
\opn\dist{dist}
\opn\cone{cone}
\opn\lex{lex}
\opn\rev{rev}
\opn\div{div} \opn\Div{Div} \opn\cl{cl} \opn\Cl{Cl}
\opn\Spec{Spec} \opn\Supp{Supp} \opn\supp{supp} \opn\Sing{Sing}
\opn\Ass{Ass} \opn\Min{Min}
\opn\Ann{Ann} \opn\Rad{Rad} \opn\Soc{Soc}
\opn\Syz{Syz} \opn\Im{Im} \opn\Ker{Ker} \opn\Coker{Coker}
\opn\Am{Am} \opn\Hom{Hom} \opn\Tor{Tor} \opn\Ext{Ext}
\opn\End{End} \opn\Aut{Aut} \opn\id{id} \opn\ini{in}

\opn\nat{nat}
\opn\pff{pf}
\opn\Pf{Pf} \opn\GL{GL} \opn\SL{SL} \opn\mod{mod} \opn\ord{ord}
\opn\Gin{Gin}
\opn\Hilb{Hilb}\opn\adeg{adeg}\opn\std{std}\opn\ip{infpt}
\opn\Pol{Pol}
\opn\sat{sat}
\opn\Var{Var}
\opn\Gen{Gen}

\opn\aff{aff} \opn\con{conv} \opn\relint{relint} \opn\st{st}
\opn\lk{lk} \opn\cn{cn} \opn\core{core} \opn\vol{vol}
\opn\link{link} \opn\star{star}
\opn\gr{gr}

\def\Ac{{\mathcal A}}

\def\Hc{{\mathcal H}}

\def\Jc{{\mathcal J}}
\def\Gc{{\mathcal G}}
\def\Fc{{\mathcal F}}
\def\Oc{{\mathcal O}}
\def\Pc{{\mathcal P}}
\def\Qc{{\mathcal Q}}

\def\Cc{{\mathcal C}}

\def\Vol{{\textnormal{Vol}}}
\def\pot#1#2{#1[\kern-0.28ex[#2]\kern-0.28ex]}

\opn\dirlim{\underrightarrow{\lim}}
\opn\inivlim{\underleftarrow{\lim}}
\let\to=\rightarrow

\def\Implies{\ifmmode\Longrightarrow \else
	\unskip${}\Longrightarrow{}$\ignorespaces\fi}
\def\implies{\ifmmode\Rightarrow \else
	\unskip${}\Rightarrow{}$\ignorespaces\fi}
\def\iff{\ifmmode\Longleftrightarrow \else
	\unskip${}\Longleftrightarrow{}$\ignorespaces\fi}

\let\:=\colon
\newtheorem{Theorem}{Theorem}[section]

\newtheorem{Corollary}[Theorem]{Corollary}
\newtheorem{Proposition}[Theorem]{Proposition}

\newtheorem{Example}[Theorem]{Example}

\let\epsilon\varepsilon
\let\phi=\varphi
\let\kappa=\varkappa
\def\qed{\ifhmode\textqed\fi
	\ifmmode\ifinner\quad\qedsymbol\else\dispqed\fi\fi}
\def\textqed{\unskip\nobreak\penalty50
	\hskip2em\hbox{}\nobreak\hfil\qedsymbol
	\parfillskip=0pt \finalhyphendemerits=0}
\def\dispqed{\rlap{\qquad\qedsymbol}}

\opn\dis{dis}
\opn\height{height}
\opn\dist{dist}
\def\pnt{{\raise0.5mm\hbox{\large\bf.}}}

\opn\Lex{Lex}
\opn\conv{conv}

\begin{document}
\title{Gorenstein Fano polytopes arising from order polytopes and chain polytopes}
\author[T. Hibi]{Takayuki Hibi}
\address[Takayuki Hibi]{Department of Pure and Applied Mathematics,
	Graduate School of Information Science and Technology,
	Osaka University,
	Toyonaka, Osaka 560-0043, Japan}
\email{hibi@math.sci.osaka-u.ac.jp}
\author[K. Matsuda]{Kazunori Matsuda}
\address[Kazunori Matsuda]{Department of Pure and Applied Mathematics,
	Graduate School of Information Science and Technology,
	Osaka University,
	Toyonaka, Osaka 560-0043, Japan}
\email{kaz-matsuda@math.sci.osaka-u.ac.jp}
\author[A. Tsuchiya]{Akiyoshi Tsuchiya}
\address[Akiyoshi Tsuchiya]{Department of Pure and Applied Mathematics,
	Graduate School of Information Science and Technology,
	Osaka University,
	Toyonaka, Osaka 560-0043, Japan}
\email{a-tsuchiya@cr.math.sci.osaka-u.ac.jp}

\date{}

\begin{abstract}
Richard Stanley introduced the order polytope $\mathcal{O}(P)$ 
and the chain polytope $\mathcal{C}(P)$ 
arising from a finite partially ordered set $P$, and showed that 
the Ehrhart polynomial of $\mathcal{O}(P)$ is equal to that of $\mathcal{C}(P)$.
In addition, the unimodular equivalence problem of $\mathcal{O}(P)$ and $\mathcal{C}(P)$
was studied by the first author and Nan Li.  
In the present paper, 
three integral convex polytopes $\Gamma(\mathcal{O}(P), -\mathcal{O}(Q))$, 
$\Gamma(\mathcal{O}(P), -\mathcal{C}(Q))$ and $\Gamma(\mathcal{C}(P), -\mathcal{C}(Q))$,
where $P$ and $Q$ are partially ordered sets with $| P | = | Q |$, will be studied. 
First, it will be shown that the Ehrhart polynomial of 
$\Gamma(\mathcal{O}(P), -\mathcal{C}(Q))$ coincides with 
that of $\Gamma(\mathcal{C}(P), -\mathcal{C}(Q))$.  Furthermore,
when $P$ and $Q$ possess a common linear extension, it will be proved that 
these three convex polytopes have the same Ehrhart polynomial. 
Second, the problem of characterizing partially ordered sets $P$ and $Q$
for which
$\Gamma(\mathcal{O}(P), -\mathcal{O}(Q))$ or 
$\Gamma(\mathcal{O}(P), -\mathcal{C}(Q))$ or  
$\Gamma(\mathcal{C}(P), -\mathcal{C}(Q))$
is a smooth Fano polytope will be solved.
Finally, when these three polytopes are smooth Fano polytopes, 
the unimodular equivalence problem of these three polytopes will be discussed.  
\end{abstract} 

\maketitle

\section*{introduction}
A convex polytope $\Pc \subset \RR^{d}$ is called {\em integral} if all of vertices of $\Pc$ belong to $\ZZ^{d}$.  
Let $\Pc \subset \RR^{d}$ be an integral convex polytope of dimension $d$. 
Given integers $n = 1, 2, \ldots,$ we define the function $i(\Pc, n)$ as follows:
\[
i(\Pc, n) := \left| (n\Pc \cap \ZZ^{d}) \right|,  
\]
where $n\Pc = \{ n\alpha \mid \alpha \in \Pc \}$. 
We call $i(\Pc, n)$ the {\em Ehrhart polynomial} of $\Pc$.  
It is known that $i(\Pc, n)$ is a polynomial in $n$ of degree $d$ with $i(\Pc, 0) = 1$ (see \cite{Ehrhart}). 

Next, we introduce some classes of Fano polytopes.
Let $\Pc \subset \RR^d$ be an integral convex polytope 
of dimension $d$.
\begin{itemize}
	\item
	We say that $\Pc$ is a {\em Fano polytope} if
	the origin of $\RR^d$ is the unique integer point
	belonging to the interior of $\Pc$. 
	\item
	A Fano polytope is called {\em Gorenstein} 
	if its dual polytope is integral.
	(Recall that the dual polytope $\Pc^\vee$
	of a Fano polytope $\Pc$ is the convex polytope
	which consists of those $x \in \RR^d$
	such that $\langle x, y \rangle \leq 1$ for all
	$y \in \Pc$, where $\langle x, y \rangle$
	is the usual inner product of $\RR^d$.)
	\item
	A {\em $\QQ$-factorial Fano polytope} is a simplicial Fano polytope,
	i.e., a Fano polytope each of whose faces is a simplex. 
	\item
	A {\em smooth Fano polytope} is a Fano polytope such that
	the vertices of each facet form a $\ZZ$-basis of $\ZZ^d$. 
\end{itemize}
Thus in particular a smooth Fano polytope is 
$\QQ$-factorial and Gorenstein.


We recall some terminologies of partially ordered sets. 
Let $P = \{p_1, \ldots, p_d\}$ be a partially ordered set. 
A {\em linear extension} of $P$ is a permutation $\sigma = i_1 i_2 \cdots i_d$ of $[d] = \{1, 2, \ldots, d\}$ 
which satisfies $i_a < i_b$ if $p_{i_a} < p_{i_b}$ in $P$. 
A subset $I$ of $P$ is called a {\em poset ideal} of $P$ if $p_{i} \in I$ and $p_{j} \in P$ together with $p_{j} \leq p_{i}$ guarantee $p_{j} \in I$.  
Note that the empty set $\emptyset$ and $P$ itself are poset ideals of $P$. 
Let $\Jc(P)$ denote the set of poset ideals of $P$.
A subset $A$ of $P$ is called an {\em antichain} of $P$ if
$p_{i}$ and $p_{j}$ belonging to $A$ with $i \neq j$ are incomparable.  
In particular, the empty set $\emptyset$ and each 1-elemant subsets $\{p_j\}$ are antichains of $P$.
Let $\Ac(P)$ denote the set of antichains of $P$.
For each subset $I \subset P$, 
we define the $(0, 1)$-vectors $\rho(I) = \sum_{p_{i}\in I} \eb_{i}$, 
where $\eb_{1}, \ldots, \eb_{d}$ are the canonical unit coordinate vectors of $\RR^{d}$.  
In particular $\rho(\emptyset)$ is the origin ${\bf 0}$ of $\RR^{d}$. 

In \cite{Stanley}, Stanley introduced the order polytope $\Oc(P)$ and the chain polytope $\Cc(P)$ 
arising from a partially ordered set $P$. 
It is known that both $\Oc(P)$ and $\Cc(P)$ are $d$-dimensional convex polytopes, and 
\[
\{ {\rm the\ sets\ of\ vertices\ of\ } \Oc(P) \} = \{ \rho(I) \mid I \in \Jc(P) \}, 
\]
\[
\{ {\rm the\ sets\ of\ vertices\ of\ } \Cc(P) \} = \{ \rho(A) \mid A \in \Ac(P) \}
\]
follows (\cite[Corollary 1.3, Theorem 2.2]{Stanley}). 
Moreover, $\Oc(P)$ and $\Cc(P)$ have the same Ehrhart polynomial (\cite[Theorem 4.1]{Stanley}). 
In particular, the volume of $\Oc(P)$ and $\Cc(P)$ are equal to $e(P) / d!$, 
where $e(P)$ is the number of linear extensions of $P$ (\cite[Corollary 4.2]{Stanley}). 

In present papers, as analogies of the order polytope and the chain polytope, 
the integral convex polytopes associated with two partially ordered sets are studied. 
These polytopes are given by combining the order polytopes and the chain polytopes. 

Let $P=\{p_1,\ldots,p_d\}$ and $Q=\{q_1,\ldots,q_d\}$ be finite partially ordered sets with $|P|=|Q|=d$.
We define integer matrices $\Psi(\Oc(P), - \Oc(Q))$, $\Psi(\Oc(P), - \Cc(Q))$ and 
$\Psi(\Cc(P), - \Cc(Q))$ as follows:
\begin{align*}
\Psi(\Oc(P), - \Oc(Q)) &= 
\{ \, \rho(I) \, \mid \, \emptyset \neq I \in \Jc(P) \, \}  \cup
\{ \, - \rho(J) \,\mid\, \emptyset \neq J \in \Jc(Q) \, \} 
\cup \{ {\bf 0} \} , \\
\Psi(\Oc(P), - \Cc(Q)) &= 
\{ \, \rho(I) \, \mid \, \emptyset \neq I \in \Jc(P) \, \}  \cup
\{ \, - \rho(J) \,\mid\, \emptyset \neq J \in \Ac(Q) \, \} 
\cup \{ {\bf 0} \} , \\
\Psi(\Cc(P), - \Cc(Q)) &= 
\{ \, \rho(I) \, \mid \, \emptyset \neq I \in \Ac(P) \, \}  \cup
\{ \, - \rho(J) \,\mid\, \emptyset \neq J \in \Ac(Q) \, \} 
\cup \{ {\bf 0} \} 
\end{align*}
and we write $\Gamma(\Oc(P), - \Oc(Q)) \subset \RR^{d}$ for the convex polytope 
which is the convex hull of $\Psi(\Oc(P), - \Cc(Q))$. 
Similarly, we define $\Gamma(\Oc(P), - \Cc(Q))$ and $\Gamma(\Cc(P), - \Cc(Q))$ 
as the convex hull of $\Psi(\Oc(P), - \Cc(Q))$ and $\Psi(\Cc(P), - \Cc(Q))$, respectively. 
These polytopes are analogies of the order polytope and the chain polytope, 
and are generalizations of the convex polytope arising from the centrally symmetric configuration 
(see \cite{CSC}). 

We note that these are $d$-dimensional polytopes.  
Moreover, since $\rho(P) = \eb_{1} + \cdots + \eb_{d} \in \Oc(P)$ and 
$\rho(\{q_j\})=\eb_{j} \in \Cc(Q)$ for $1 \leq j \leq d$, 
we have that the origin ${\bf 0}$ of $\RR^{d}$
is belonging to the interior of $\Gamma(\Oc(P), - \Cc(Q))$ and that of $\Gamma(\Cc(P), - \Cc(Q))$. 
However, it is not necessarily that $\Gamma(\Oc(P), - \Oc(Q))$ has the same property. 
It is known that the origin ${\bf 0}$ of $\RR^{d}$
is belonging to the interior of $\Gamma(\Oc(P), - \Oc(Q))$ if and only if $P$ and $Q$ possess a 
common linear extension (\cite[Lemma 1.1]{twin}). 
In addition, it is also known that $\Gamma(\Oc(P), - \Oc(P))$, $\Gamma(\Oc(P), - \Cc(Q))$ and
$\Gamma(\Cc(P), - \Cc(Q))$ are always Gorenstein Fano
(\cite[Corollary 2.3]{HMOS}, \cite[Corollary 1.2]{orderchain}, \cite[Theorem 2.8]{harmony}) and 
$\Gamma(\Oc(P), - \Oc(Q))$ is Gorenstein Fano if and only if  
$P$ and $Q$ possess a common linear extension (\cite[Corollary 2.2]{twin}). 
Hence to determine when these polytopes are smooth Fano is an important problem. 
Similarly, the question whether these polytopes are unimodularly equivalent when these polytopes are smooth
is also interesting. 
The problem when $\Oc(P)$ and $\Cc(P)$ are unimodularly equivalent was solved in \cite{unimodular}. 

This paper is organized as follows. 
In Section $1$, we study the Ehrhart polynomials of these polytopes 
(Theorem \ref{Ehrhart}).
This is an analogy of Stanley's results mentioned before. 
In Section $2$, we study the characterization problem of partially ordered sets yield smooth Fano polytopes (Theorems \ref{ccs}, \ref{ocs} and \ref{oos}).
Finally, in Section $3$, we study the unimodular equivalence of smooth Fano polytopes.
In fact, we show $\Gamma(\Oc(P), - \Oc(Q))$ and $\Gamma(\Cc(P), - \Cc(Q))$ are unimodularly  equivalent, however,   these polytopes are not unimodularly equivalent to $\Gamma(\Oc(P), - \Cc(Q))$, when all polytopes are smooth (Theorem \ref{equi}).

For fundamental materials on Gr\"{o}bner bases and toric ideals, see \cite{dojoEN}. 

 \section{Squarefree Quadratic Gr\"{o}bner basis and Ehrhart polynomial}
 
In this section, we show the following:

\begin{Theorem}
	\label{Ehrhart}
Work with the same notation as in Introduction. 
Then we have
\[
i(\Gamma(\Oc(P), - \Cc(Q)), n) = i(\Gamma(\Cc(P), - \Cc(Q)), n).
\] 
In particular, the volume of $\Gamma(\Oc(P), - \Cc(Q))$ is the same as that of $\Gamma(\Cc(P), - \Cc(Q))$. 
Moreover, if $P$ and $Q$ possess a common linear extension, then we have
\[
i(\Gamma(\Oc(P), - \Oc(Q)), n) = i(\Gamma(\Oc(P), - \Cc(Q)), n) = i(\Gamma(\Cc(P), - \Cc(Q)), n).
\]
In this case, these polytopes have the same volume. 
\end{Theorem}

In order to prove this theorem, we use the following facts. 
We say that an integral convex polytope $\Pc \subset \RR^{d}$ is {\em normal} if, 
for each integer $N > 0$ 
and for each ${\bf a} \in N \Pc \cap \ZZ^{d}$, there exist ${\bf a}_1, \ldots, {\bf a}_N \in \Pc \cap \ZZ^{d}$ 
such that ${\bf a} = {\bf a}_1 + \cdots +{\bf a}_N$.

\begin{itemize}
	\item Let $\Pc \subset \RR^d$ be an integral convex polytope of $\dim \Pc = d$ and 
	\[
	K[\Pc] := K[x_1^{\alpha_1} \cdots x_d^{\alpha_d}t \mid (\alpha_1, \ldots, \alpha_d) \in \Pc \ ] \subset K[x_1, \ldots, x_d, t]
	\]
	be the toric ring of $\Pc$ over a field $K$. 
	 Assume that there exists a monomial order $<$ on $K[x_1, \ldots, x_d, t]$ such that the initial ideal 
	 $\mathrm{in}_{<}(I_{\Pc})$ of the toric ideal of $K[\Pc]$ with respect to the order $<$ is squarefree. 
	 Then $\Pc$ is a normal polytope.  (see \cite{dojoEN}) \\
	 \item Let $\Pc \subset \RR^d$ be an integral convex polytope. 
	 If $\Pc$ is normal, then the Ehrhart polynomial of $\Pc$ is equal to the Hilbert polynomial of 
	 the toric ring $K[\Pc]$.  \\
	 \item Let $S$ be a polynomial ring and $I \subset S$ be a graded ideal of $S$. 
	 Let $<$ be a monomial order on $S$. 
	 Then $S/I$ and $S/\mathrm{in}_{<}(I)$ have the same Hilbert function. 
	 (see \cite[Corollary 6.1.5]{Monomial})
\end{itemize}

\vspace{2mm}

At first, we define the toric rings $K[\Gamma(\Oc(P), - \Oc(Q))]$ (resp. $K[\Gamma(\Oc(P), - \Cc(Q))]$ 
and $K[\Gamma(\Cc(P), - \Cc(Q))]$) of the polytope 
$\Gamma(\Oc(P), - \Oc(Q))$ (resp. $\Gamma(\Oc(P), - \Cc(Q))$ and $\Gamma(\Oc(P), - \Oc(Q))$), 
and prove the normality of these polytopes by using the theory of toric ideals. 

Let, as before, $P = \{ p_{1}, \ldots, p_{d} \}$ and $Q = \{ q_{1}, \ldots, q_{d} \}$ 
be finite partially ordered sets with $|P| = |Q| = d$. 
For a poset ideal of $P$ or $Q$, we write $\max(I)$ for the set of maximal elements of $I$.
In particular, $\max(I)$ is an antichain.  
Note that for each antichain $A$, there exists a poset ideal $I$ such that 
$A=\max(I)$. 

Let $K[{\bf t}, {\bf t}^{-1}, s] 
= K[t_{1}, \ldots, t_{d}, t_{1}^{-1}, \ldots, t_{d}^{-1}, s]$
denote the Laurent polynomial ring in $2d + 1$ variables over a field $K$. 
If $\alpha = (\alpha_{1}, \ldots, \alpha_{d}) \in \ZZ^{d}$, then
${\bf t}^{\alpha}s$ is the Laurent monomial
$t_{1}^{\alpha_{1}} \cdots t_{d}^{\alpha_{d}}s \in K[{\bf t}, {\bf t}^{-1}, s]$. 
In particular ${\bf t}^{\bf 0}s = s$.
Then we define the toric rings $K[\Gamma(\Oc(P), - \Oc(Q))], K[\Gamma(\Oc(P), - \Cc(Q))]$ 
and $K[\Gamma(\Cc(P), - \Cc(Q))]$ as follows: 
\[
K[\Gamma(\Oc(P), - \Oc(Q))] = K[{\bf t}^{\alpha}s \mid \alpha \in \Gamma(\Oc(P), - \Oc(Q))], 
\]
\[
K[\Gamma(\Oc(P), - \Cc(Q))] = K[{\bf t}^{\alpha}s \mid \alpha \in \Gamma(\Oc(P), - \Cc(Q))], 
\]
\[
K[\Gamma(\Cc(P), - \Cc(Q))] = K[{\bf t}^{\alpha}s \mid \alpha \in \Gamma(\Cc(P), - \Cc(Q))]. 
\]

Let 
\begin{eqnarray*}
K[{\Oc \Oc}] &=& K[\{x_{I}\}_{\emptyset \neq I \in \Jc(P)} \cup 
\{y_{J}\}_{\emptyset \neq J \in \Jc(Q)} \cup \{ z \}], \\
K[{\Oc \Cc}] &=& K[\{x_{I}\}_{\emptyset \neq I \in \Jc(P)} \cup 
\{y_{\max(J)}\}_{\emptyset \neq J \in \Jc(Q)} \cup \{ z \}], \\
K[{\Cc \Cc}] &=& K[\{x_{\max(I)}\}_{\emptyset \neq I \in \Jc(P)} \cup 
\{y_{\max(J)}\}_{\emptyset \neq J \in \Jc(Q)} \cup \{ z \}]
\end{eqnarray*}
denote the polynomial rings over $K$, 
and define the surjective ring homomorphisms 
$\pi_{\Oc \Oc}$, $\pi_{\Oc \Cc}$ and $\pi_{\Cc \Cc}$ by the following: 

\vspace{2mm}

\begin{itemize}
	\item $\pi_{\Oc \Oc} : K[{\Oc \Oc}] \to K[\Gamma(\Oc(P), - \Oc(Q))]$ by setting \\
	$\pi_{\Oc \Oc}(x_{I}) = {\bf t}^{\rho(I)}s$, $\pi_{\Oc \Oc}(y_{J}) = {\bf t}^{- \rho(J)}s$ and 
	$\pi_{\Oc \Oc}(z) = s$, \\
	\item $\pi_{\Oc \Cc} : K[{\Oc \Cc}] \to K[\Gamma(\Oc(P), - \Cc(Q))]$ by setting \\
	$\pi_{\Oc \Cc}(x_{I}) = {\bf t}^{\rho(I)}s$, $\pi_{\Oc \Cc}(y_{\max(J)}) = {\bf t}^{- \rho(\max(J))}s$ and 
	$\pi_{\Oc \Cc}(z) = s$, \\
	\item $\pi_{\Cc \Cc} : K[{\Cc \Cc}] \to K[\Gamma(\Cc(P), - \Cc(Q))]$ by setting \\
	$\pi_{\Cc \Cc}(x_{\max(I)}) = {\bf t}^{\rho(\max(I))}s$, $\pi_{\Cc \Cc}(y_{\max(J)}) = {\bf t}^{- \rho(\max(J))}s$ and $\pi_{\Cc \Cc}(z) = s$ \\
\end{itemize}
where $\emptyset \neq I \in \Jc(P)$ and $\emptyset \neq J \in \Jc(Q)$. 
Then the {\em toric ideal} $I_{\Gamma(\Oc(P), - \Oc(Q))}$ of $\Gamma(\Oc(P), - \Oc(Q))$ is 
the kernel of $\pi_{\Oc \Oc}$. 
Similarly, the toric ideal $I_{\Gamma(\Oc(P), - \Cc(Q))}$ (resp. $I_{\Gamma(\Cc(P), - \Cc(Q))}$) is 
the kernel of $\pi_{\Oc \Cc}$ (resp. $\pi_{\Cc \Cc}$). 


\vspace{2mm}

Next, we introduce monomial orders $<_{\Oc \Oc}$, $<_{\Oc \Cc}$ and $<_{\Cc \Cc}$
and $\Gc_{\Oc \Oc}$, $\Gc_{\Oc \Cc}$ and $\Gc_{\Cc \Cc}$ which are the set of binomials. 

Let $<_{\Oc \Oc}$ denote a reverse lexicographic order on $K[{\Oc \Oc}]$
satisfying
\begin{itemize}
	\item
	$z <_{\Oc \Oc}  y_{J} <_{\Oc \Oc} x_{I}$;
	\item
	$x_{I'} <_{\Oc \Oc} x_{I}$ if $I' \subset I$;
	\item
	$y_{J'} <_{\Oc \Oc} y_{J}$ if $J' \subset J$,
\end{itemize}
and $\Gc_{\Oc \Oc}$ the set of the following binomials:
\begin{enumerate}
	\item[(i)]
	$x_{I}x_{I'} - x_{I\cup I'}x_{I \cap I'}$;
	\item[(ii)]
	$y_{J}y_{J'} - y_{J\cup J'}y_{J \cap J}$;
	\item[(iii)]
	$x_{I}y_{J} - x_{I \setminus \{p_{i}\}}y_{J \setminus \{q_{i}\}}$,
\end{enumerate}
and let $<_{\Oc \Cc}$ denote a reverse lexicographic order on $K[{\Oc \Cc}]$
satisfying
\begin{itemize}
	\item
	$z <_{\Oc \Cc}  y_{\max(J)} <_{\Oc \Cc} x_{I}$;
	\item
	$x_{I'} <_{\Oc \Cc} x_{I}$ if $I' \subset I$;
	\item
	$y_{\max(J')} <_{\Oc \Cc} y_{\max(J)}$ if $J' \subset J$,
\end{itemize}
and $\Gc_{\Oc \Cc}$ the set of the following binomials:
\begin{enumerate}
	\item[(iv)]
	$x_{I}x_{I'} - x_{I\cup I'}x_{I \cap I'}$;
	\item[(v)]
	$y_{\max(J)}y_{\max(J')} - y_{\max(J\cup J')}y_{\max(J * J')}$;
	\item[(vi)]
	$x_{I}y_{\max(J)} - x_{I \setminus \{p_{i}\}}y_{\max(J) \setminus \{q_{i}\}}$,
\end{enumerate}
and let $<_{\Cc \Cc}$ denote a reverse lexicographic order on $K[{\Cc \Cc}]$
satisfying
\begin{itemize}
	\item
	$z <_{\Cc \Cc}  y_{\max(J)} <_{\Cc \Cc} x_{\max(I)}$;
	\item
	$x_{\max(I')} <_{\Cc \Cc} x_{\max(I)}$ if $I' \subset I$;
	\item
	$y_{\max(J')} <_{\Cc \Cc} y_{\max(J)}$ if $J' \subset J$,
\end{itemize}
and $\Gc_{\Cc \Cc}$ the set of the following binomials:
\begin{enumerate}
	\item[(vii)]
	$x_{\max(I)}x_{\max(I')} - y_{\max( I \cup I')}y_{\max(I * I')}$;
	\item[(viii)]
	$y_{\max(J)}y_{\max(J')} - y_{\max(J\cup J')}y_{\max(J * J')}$;
	\item[(ix)]
	$x_{\max(I)}y_{\max(J)} - x_{\max(I) \setminus \{p_{i}\}}y_{\max(J) \setminus \{q_{i}\}}$,
\end{enumerate}
where 
\begin{itemize}
	\item
	$x_{\emptyset} = y_{\emptyset} = z$;
	\item
	$I$ and $I'$ are poset ideals of $P$ which are incomparable in $\Jc(P)$;
	\item  
	$J$ and $J'$ are poset ideals of $Q$ which are incomparable in $\Jc(Q)$;
	\item
	$I * I'$ is the poset ideal of $P$ generated by $\max(I \cap I')\cap (\max(I)\cup \max(I'))$;
	\item
	$J * J'$ is the poset ideal of $Q$ generated by $\max(J \cap J')\cap (\max(J)\cup \max(J'))$;
	\item $p_{i}$ is a maximal element of $I$ and $q_{i}$ is a maximal element of $J$. 
\end{itemize}

It is known that $\Gc_{\Oc \Oc}$ is a Gr\"{o}bner basis of $I_{\Gamma(\Oc(P), - \Oc(Q))}$ 
with respect to $<_{\Oc \Oc}$(\cite[Theorem 2.1]{twin}) and $\Gc_{\Oc \Cc}$ is a Gr\"{o}bner basis of $I_{\Gamma(\Oc(P), - \Cc(Q))}$ 
with respect to $<_{\Oc \Cc}$(\cite[Theorem 1.1]{orderchain}). 
As corollaries, we have the normality of $\Gamma(\Oc(P), - \Oc(Q))$ and $\Gamma(\Oc(P), - \Cc(Q))$ 
(\cite[Corollary 2.2]{twin}, \cite[Corollary 1.2]{orderchain}). 
However, the polytope $\Gamma(\Cc(P), - \Cc(Q))$ was studied in \cite{harmony}
in more general situation, the normality of this polytope is still open. 
To prove that $\Gamma(\Cc(P), - \Cc(Q))$ is normal,  
we compute a Gr\"{o}bner basis of $I_{\Gamma(\Cc(P), - \Cc(Q))}$ with respect to $<_{\Cc \Cc}$. 
 
\vspace{2mm}

	\begin{Proposition}
		Work with the same situation as above.
		Then $\Gc_{\Cc \Cc}$ is a Gr\"obner basis of
		$I_{\Gamma(\Cc(P), - \Cc(Q))}$ with respect to $<_{\Cc \Cc}$.
	\end{Proposition}
	\begin{proof}
		First, it is clear that $\Gc_{\Cc \Cc} \subset I_{\Gamma(\Cc(P), - \Cc(Q))}$.
		For a binomial $f = u - v$, we call $u$ the {\em first}
		monomial of $f$ and we call $v$ the {\em second} monomial of $f$. 
		By the definition of $<_{\Cc \Cc}$, the initial monomial of each of the binomials (vii) -- (ix) 
		with respect to $<_{\Cc \Cc}$ is its first monomial. 
		Let ${\rm in}_{<_{\Cc \Cc}}(\Gc_{\Cc \Cc})$ denote the set of initial monomials of binomials 
		belonging to $\Gc_{\Cc \Cc}$.  
		From \cite[(0.1)]{OHrootsystem}, it follows that,
		in order to show that $\Gc_{\Cc \Cc}$ is a Gr\"obner basis of
		$I_{\Gamma(\Cc(P), - \Cc(Q))}$ with respect to $<_{\Cc \Cc}$, we need to prove 
		the following:
		
		($\clubsuit$) If $u$ and $v$ are monomials belonging to 
		$K[{\Cc \Cc}]$ with $u \neq v$ such that 
		$u \not\in \langle {\rm in}_{<_{\Cc \Cc}}(\Gc_{\Cc \Cc}) \rangle$ 
		and $v \not\in \langle {\rm in}_{<_{\Cc \Cc}}(\Gc_{\Cc \Cc}) \rangle$,
		then $\pi_{\Cc \Cc}(u) \neq \pi_{\Cc \Cc}(v)$.
		
		\vspace{2mm}
		
		Let $u, v \in K[{\Cc \Cc}]$ be monomials
		with $u \neq v$.  Write
		\[
		u = z^{\alpha} x_{\max(I_{1})}^{\xi_{1}} \cdots x_{\max(I_{a})}^{\xi_{a}}
		y_{\max(J_{1})}^{\nu_{1}} \cdots y_{\max(J_{b})}^{\nu_{b}},
		\]
		\[
		v = z^{\alpha'} x_{\max(I'_{1})}^{\xi'_{1}} \cdots x_{\max(I'_{a'})}^{\xi'_{a'}}
		y_{\max(J'_{1})}^{\nu'_{1}} \cdots y_{\max(J'_{b'})}^{\nu'_{b'}},
		\]
		where
		\begin{itemize}
			\item
			$\alpha \geq 0$, $\alpha' \geq 0$;
			\item
			$I_{1}, \ldots, I_{a}, I'_{1}, \ldots, I'_{a'} 
			\in \Jc(P) \setminus \{ \emptyset \}$;
			\item
			$J_{1}, \ldots, J_{b}, J'_{1}, \ldots, J'_{b'} 
			\in \Jc(Q) \setminus \{ \emptyset \}$;
			\item
			$\xi_{1}, \ldots, \xi_{a}, 
			\nu_{1}, \ldots, \nu_{b},
			\xi'_{1}, \ldots, \xi'_{a'}, 
			\nu'_{1}, \ldots, \nu'_{b'} > 0$,
		\end{itemize}
		and where $u$ and $v$ are relatively prime with 
		$u \not\in \langle {\rm in}_{<_{\Cc \Cc}}(\Gc_{\Cc \Cc}) \rangle$ 
		and $v \not\in \langle {\rm in}_{<_{\Cc \Cc}}(\Gc_{\Cc \Cc}) \rangle$.
		Note that either $\alpha = 0$ or $\alpha' = 0$ since $u$ and $v$ are relatively prime.
		Hence we may assume that $\alpha' = 0$.  Thus
		\[
		u = z^{\alpha} x_{\max(I_{1})}^{\xi_{1}} \cdots x_{\max(I_{a})}^{\xi_{a}}
		y_{\max(J_{1})}^{\nu_{1}} \cdots y_{\max(J_{b})}^{\nu_{b}},
		\]
		\[
		v = x_{\max(I'_{1})}^{\xi'_{1}} \cdots x_{\max(I'_{a'})}^{\xi'_{a'}}
		y_{\max(J'_{1})}^{\nu'_{1}} \cdots y_{\max(J'_{b'})}^{\nu'_{b'}}. 
		\]
		Since the initial monomial of each of the binomials (vii) -- (ix) with respect to $<_{\Cc \Cc}$ 
		does not belong to $\langle {\rm in}_{<_{\Cc \Cc}}(\Gc_{\Cc \Cc}) \rangle$, 
		we have that
		\begin{itemize}
			\item
			$I_{1} \subsetneq  I_{2} \subsetneq \cdots \subsetneq I_{a}$ and $J_{1} \subsetneq J_{2} \subsetneq \cdots \subsetneq J_{b}$;
			\item
			$I'_{1} \subsetneq I'_{2} \subsetneq \cdots \subsetneq I'_{a'}$ and	$J'_{1} \subsetneq J'_{2} \subsetneq \cdots \subsetneq J'_{b'}$.
		\end{itemize}
		Furthermore, by virtue of \cite{Hibi1987} and \cite{chain}, it suffices to discuss 
		$u$ and $v$ with $(a, a') \neq (0, 0)$ and $(b, b') \neq (0,0)$.
		
		Since $I_a \neq I'_{a'}$, we may assume that $I_a \setminus I'_{a'} \neq \emptyset$.
		Then there exists a maximal element $p_{i^*}$ of $I_a$ with $p_{i^*} \notin I'_{a'}$. 
		
		Suppose that $\pi_{\Cc \Cc}(u)=\pi_{\Cc \Cc}(v)$.
		Then we have 
		$$\sum\limits_{\stackrel{I \in \{I_1,\ldots,I_a\}}{p_i \in \max(I)}}\xi_I-\sum\limits_{\stackrel{J \in \{J_1,\ldots,J_b\}}{q_i \in \max(J)}}\nu_J
		=\sum\limits_{\stackrel{I' \in \{I'_1,\ldots,I'_{a'}\}}{p_i \in \max(I)'}}\xi'_{I'}-\sum\limits_{\stackrel{J' \in \{J'_1,\ldots,J'_{b'}\}}{q_i \in \max(J')}}\nu'_{J'}.$$
		for all $1 \leq i \leq d$ by comparing the degree of $t_i$. 
		By assumption, $p_{i^*} \notin I'_{a'}$. 
		This means that $p_{i^*} \notin \max(I'_{c'})$ for all $1 \le c' \le a'$.   
		Hence we have
		 $$\sum\limits_{\stackrel{I \in \{I_1,\ldots,I_a\}}{p_{i^*} \in \max(I)}}\xi_I-\sum\limits_{\stackrel{J \in \{J_1,\ldots,J_b\}}{q_{i^*} \in \max(J)}}\nu_J
		 =-\sum\limits_{\stackrel{J' \in \{J'_1,\ldots,J'_{b'}\}}{q_{i^*} \in \max(J')}}\nu'_{J'}\leq 0.$$
		Moreover,  since $p_{i^*}$ is a maximal element of $I_a$, we also have 
		$$\sum\limits_{\stackrel{I \in \{I_1,\ldots,I_a\}}{p_{i^*} \in \max(I)}}\xi_I>0. $$
		Hence there exists an integer $c$ with $1\leq c \leq b$ such that $q_{i^*} \in \max(J_c)$.
		Therefore we have 
		$x_{\max(I_a)}y_{\max(J_c)} \in \langle{\rm in}_{<_{\Cc \Cc}}(\Gc_{\Cc \Cc}) \rangle$, 
		but this is a contradiction.
\end{proof}

By this proposition, it is clear that the initial ideal $\mathrm{in}_{<_{\Cc \Cc}}(I_{\Gamma(\Cc(P), - \Cc(Q))})$ of the toric ideal $I_{\Gamma(\Cc(P), - \Cc(Q))}$ with respect to the order $<_{\Cc \Cc}$ is squarefree. 
Hence we have 

\begin{Corollary}
$\Gamma(\Cc(P), - \Cc(Q))$ is a normal Gorenstein Fano polytope for any partially ordered sets $P$ and $Q$ with $|P| = |Q| = d$. 
\end{Corollary}

Here, we put
\[
R_{\Oc \Oc} := K[\Oc \Oc]/\mathrm{in}_{<_{\Oc \Oc}}(I_{\Gamma(\Oc(P), - \Oc(Q))}), 
\]
\[
R_{\Oc \Cc} := K[\Oc \Cc]/\mathrm{in}_{<_{\Oc \Cc}}(I_{\Gamma(\Oc(P), - \Cc(Q))}), 
\]
\[
R_{\Cc \Cc} := K[\Cc \Cc]/\mathrm{in}_{<_{\Cc \Cc}}(I_{\Gamma(\Cc(P), - \Cc(Q))}). 
\]
Next, we prove the following. 

\begin{Proposition}
The ring $R_{\Oc \Cc}$ is isomorphic to the ring $R_{\Cc \Cc}$ 
for any partially ordered sets $P$ and $Q$ with $|P| = |Q| = d$. 
Moreover, if $P$ and $Q$ possess a common linear extension, then these rings 
$R_{\Oc \Oc}$, $R_{\Oc \Cc}$ and $R_{\Cc \Cc}$ are isomorphic. 
\end{Proposition}
\begin{proof}
From \cite[Theorem 2.1]{twin}, \cite[Theorem 1.1]{orderchain} and Proposition 1.2, we have
\[
R_{\Oc \Oc} \cong \frac{K[\{x_{I}\}_{\emptyset \neq I \in \Jc(P)} \cup 
\{y_{J}\}_{\emptyset \neq J \in \Jc(Q)} \cup \{ z \}]}
{(x_{I}x_{I'}, y_{J}y_{J'}, x_{I}y_{J} \mid I, I', J \ \mathrm{and}\ J' \ \mathrm{satisfy \ (*)}) }, 
\]
\[
R_{\Oc \Cc} \cong \frac{K[\{x_{I}\}_{\emptyset \neq I \in \Jc(P)} \cup 
\{y_{\max(J)}\}_{\emptyset \neq J \in \Jc(Q)} \cup \{ z \}]}
{(x_{I}x_{I'}, y_{\max(J)}y_{\max(J')}, x_{I}y_{\max(J)} \mid I, I', J \ \mathrm{and}\ J' \ \mathrm{satisfy \ (*)}) }, 
\]
\[
R_{\Cc \Cc} \cong \frac{K[\{x_{\max(I)}\}_{\emptyset \neq I \in \Jc(P)} \cup 
\{y_{\max(J)}\}_{\emptyset \neq J \in \Jc(Q)} \cup \{ z \}]}
{(x_{\max(I)}x_{\max(I')}, y_{\max(J)}y_{\max(J')}, x_{\max(I)}y_{\max(J)} \mid I, I', J \ \mathrm{and}\ J' \ \mathrm{satisfy \ (*)}) }, 
\]
where the condition $(*)$ is the following: 
\begin{itemize}
	\item $I$ and $I'$ are poset ideals of $P$ which are incomparable in $\Jc(P)$;
	\item $J$ and $J'$ are poset ideals of $Q$ which are incomparable in $\Jc(Q)$;
	\item There exists $1 \le i \le d$ such that $p_{i}$ is a maximal element of $I$ and $q_{i}$ is a maximal element of $J$. 
\end{itemize}
Hence it is easy to see that 
the ring homomorphism $\phi : R_{\Oc \Cc} \to R_{\Cc \Cc}$ by setting $\phi(x_I) = x_{\max(I)}$, 
$\phi(y_{\max(J)}) = y_{\max(J)}$ and $\phi(z) = z$ is an isomorphism. 
Similarly, if $P$ and $Q$ possess a common linear extension, 
we can see that the ring homomorphism $\phi^{'} : R_{\Oc \Oc} \to R_{\Oc \Cc}$ by setting $\phi^{'}(x_I) = x_I$, 
$\phi^{'}(y_J) = y_{\max(J)}$ and $\phi^{'}(z) = z$ is an isomorphism.
Hence we have $R_{\Oc \Oc} \cong R_{\Oc \Cc} \cong R_{\Cc \Cc}$. 
\end{proof} 

Now, we can prove Theorem 1.1. 
\begin{proof}[Proof of Theorem 1.1]
From \cite[Corollarly 1.2]{orderchain} and Proposition 1.2, we have that 
both $\Gamma(\Oc(P), - \Cc(Q))$ and $\Gamma(\Cc(P), - \Cc(Q))$ are normal. 
Hence the Ehrhart polynomial of $\Gamma(\Oc(P), - \Cc(Q))$ (resp. $\Gamma(\Cc(P), - \Cc(Q))$) 
is equal to the Hilbert polynomial of 
$K[\Gamma(\Oc(P), - \Cc(Q))]$ (resp. $K[\Gamma(\Cc(P), - \Cc(Q))]$). 
By Proposition 1.4, $R_{\Oc \Cc}$ and $R_{\Cc \Cc}$ have the same Hilbert polynomial. 
Hence $K[\Gamma(\Oc(P), - \Cc(Q))]$ and $K[\Gamma(\Cc(P), - \Cc(Q))]$ 
also have the same Hilbert polynomial. 
Therefore we have the desired conclusion. 

If $P$ and $Q$ possess a common linear extension, $\Gamma(\Oc(P), - \Oc(Q))$ is also normal 
from \cite[Corollarly 2.2]{twin}. 
Therefore, by the same argument, we have the desired conclusion. 
\end{proof}

We immidiately obtain the following corollary.
\begin{Corollary}
	For any partially ordered sets $P$ and $Q$ with $|P| = |Q| = d$,
	 we have 
	 $$i(\Gamma(\Oc(P), - \Cc(Q)), n)=i(\Gamma(\Cc(Q), - \Cc(P)), n).$$
	 In particular, these polytopes have the same volume.
\end{Corollary}

As the end of this section, 
we give an example that $P$ and $Q$ do not have any common linear extension.
\begin{Example}
Let $P = \{p_1 < p_2\}$ and $Q = \{q_2 < q_1\}$ be chains. 
It is clear that $P$ and $Q$ have no common linear extension. 
Then 
\[
i(\Gamma(\Oc(P), - \Oc(Q)), n) = \frac{3}{2} n^2 + \frac{5}{2} n + 1, 
\] 
\[
i(\Gamma(\Oc(P), - \Cc(Q)), n) = i(\Gamma(\Cc(P), - \Cc(Q)), n) = 2 n^2 + 2 n + 1. 
\]
\end{Example}

\section{when are three polytopes smooth?}

In this section, 
we consider the characterization problem of partially ordered sets yield smooth Fano polytopes. 

First, we recall some definitions.  
Let $\Pc \subset \RR^d$ be a Fano polytope.
\begin{itemize}
	\item We call $\Pc$ \textit{centrally symmetric} if $\Pc= -\Pc$. 
	\item We call $\Pc$ \textit{pseudo-symmetric} if there exists a facet $\Fc$ of $\Pc$
	such that $-\Fc$ is also its facet.
	Note that every centrally symmetric polytope is pseudo-symmetric.
	\item A \textit{del Pezzo polytope} of dimension $2k$ is a convex polytope
	$$V_{2k}=\text{conv}(\pm \eb_1,\ldots\pm \eb_{2k},\pm(\eb_1+\cdots +\eb_{2k})).$$
	Note that del Pezzo polytopes are centrally symmetric smooth Fano polytopes.
	\item A \textit{pseudo del Pezzo polytope} of dimension $2k$ is a convex polytope
	$$\tilde{V}_{2k}=\text{conv}(\pm \eb_1,\ldots\pm \eb_{2k},\eb_1+\cdots +\eb_{2k}).$$
	Note that pseudo del Pezzo polytopes are pseudo-symmetric smooth Fano polytopes.
	\item Let us that $\Pc$ \textit{splits} into $\Pc_1$ and $\Pc_2$ if the convex hull of two Fano polytopes
	$\Pc_1 \subset \RR^{d_1}$ and $\Pc_2 \subset \RR^{d_2}$ with $d=d_1+d_2$, i.e., by renumbering
	$$\Pc=\text{conv}(\{(\alpha_1,0),(0,\alpha_2) \in \RR^d : \alpha_1 \in \Pc_1,\alpha_2 \in \Pc_2\}).$$
	Then we write $\Pc=\Pc_1 \oplus \Pc_2$.
\end{itemize}

There is well-known fact on the characterization of centrally symmetric or pseudo-symmetric smooth Fano polytopes.
\begin{itemize}
	\item Any centrally symmetric smooth Fano polytope splits into copies of the closed interval $[-1,1]$ or a del Pezzo polytope \cite{variety}.
	\item Any pseudo-symmetric smooth Fano polytope splits into copies of the closed interval $[-1,1]$ or a del Pezzo polytope or pseudo del Pezzo polytope \cite{class,variety}.
\end{itemize}
	
Let $P$ and $Q$ be partially ordered sets with $|P|=|Q|=d$.
In this section, we consider when each of $\Gamma(\Oc(P), - \Oc(Q))$, $\Gamma(\Oc(P), - \Cc(Q))$ 
and $\Gamma(\Cc(P), - \Cc(Q))$ is smooth.

First, we consider when $\Gamma(\Cc(P), - \Cc(Q))$ is smooth.
For $1 \leq i \leq d$, we set $A_i(P)=\{I \in A(P) \: |A|=i\}.$
\begin{Theorem}
	\label{ccs}
	For $d \geq 2$, let $P$ and $Q$ be partially ordered sets  with $|P|=|Q|=d$.
	Then the following conditions are equivalent:
	\begin{enumerate}
		\item[(i)] $\Gamma(\Cc(P), - \Cc(Q))$ is $\QQ$-factorial;
		\item[(ii)] $\Gamma(\Cc(P), - \Cc(Q))$ is smooth;
		\item[(iii)] $\Gamma(\Cc(P), - \Cc(Q))$ splits into copies of the closed interval $[-1,1]$ or a del Pezzo $2$-polytope or pseudo del Pezzo $2$-polytope;
		\item[(iv)] For any $I_1,I_2 \in A_2(P)$ with $I_1 \neq I_2$,
		$I_1 \cap I_2=\emptyset$ and for any $J_1,J_2 \in A_2(Q)$ with $J_1 \neq J_2$,
		$J_1 \cap J_2=\emptyset$, and for any $I \in A_2(P)$ and for any $J \in A_2(Q)$, $|I \cap J| \neq 1$.
	\end{enumerate}	
\end{Theorem}
\begin{proof}
	\textbf{((i) $\Rightarrow$ (iv))}
	Let $p_{i_1} \prec p_{i_2} \prec \cdots \prec p_{i_s}$ be a maximal chain of $P$.
	Then $x_{i_1}+x_{i_2}+\cdots+x_{i_s}=1$ is a facet of $\Cc(P)$,
	in particular, this is a facet of $\Gamma(\Cc(P), - \Cc(Q))$.
	Since $\Gamma(\Cc(P), - \Cc(Q))$ is simplicial,
	this facet is a ($d-1$)-simplex.
	Hence there exist just $d-s$ antichains $I_1,\ldots,I_{d-s} \in A(P)\setminus A_1(P)$ such that
	for each $I_k$, $|\{p_{i_1},p_{i_2},\ldots,p_{i_s}\}\cap I_k|=1$.
	Since for each $j \in [d]\setminus \{p_{i_1},p_{i_2},\ldots,p_{i_s}\}$,
	there exists $i \in \{p_{i_1},p_{i_2},\ldots,p_{i_s}\}$ such that $\{i,j\}$ is an antichain of $P$,
	for each $j \in [d]\setminus \{p_{i_1},p_{i_2},\ldots,p_{i_s}\}$,
	there exists just one $i \in \{p_{i_1},p_{i_2},\ldots,p_{i_s}\}$ such that $\{i,j\}$ is an antichain of $P$.
	Then for $k \geq 3$, $A_k(P)=\emptyset$. 	 
	
	First, we assume that there exist $I_1,I_2 \in A_2(P)$ with $I_1 \neq I_2$ such that
	$I_1 \cap I_2 \neq \emptyset$.
	Let $I_1=\{p_{i_1},p_{i_2}\}$ and $I_2=\{p_{i_1},p_{i_3}\}$.
	Then we know that $\{p_{i_2},p_{i_3}\}$ is not an antichain of $P$.
	Indeed, if $\{p_{i_2},p_{i_3}\}$ is an antichain of $P$, then $\{p_{i_1},p_{i_2},p_{i_3}\}$ is also an antichain of $P$.
	Hence there exists a maximal chain $p_{j_1} \prec p_{j_2} \prec \cdots \prec p_{j_t}$ of $P$ such that $\{p_{i_2},p_{i_3}\} \subset \{p_{j_1},p_{j_2}\ldots,p_{j_t}\}$.
	Then since $\{p_{i_1},p_{i_2}\}$ and $\{p_{i_1},p_{i_3}\}$ are antichains of $P$,
	a facet $x_{j_1}+x_{j_2}+\cdots+x_{j_t}=1$ of $\Gamma(\Cc(P), - \Cc(Q))$ is not a ($d-1$)-simplex.
	
	Next, we assume that for any $I_1,I_2 \in A_2(P)$ with $I_1 \neq I_2$,
	$I_1 \cap I_2=\emptyset$, and for any $J_1,J_2 \in A_2(Q)$ with $J_1 \neq J_2$,
	$J_1 \cap J_2=\emptyset$, and there exist $I \in A_2(P)$ and $J \in A_2(Q)$ such that $|I \cap J| = 1$.
	We let $I=\{p_{i_1},p_{i_2}\}$ and $J=\{q_{i_1},q_{i_3}\}$.
	Then $x_{i_2}-x_{i_3}=1$ is a face of $\Gamma(\Cc(P), - \Cc(Q))$ and this face is not simplex.
	Indeed, we set $\Hc=\{(x_1,\ldots,x_d) \in \RR^d \: x_{i_2}-x_{i_3}=1\}$ and $\Hc^+=\{(x_1,\ldots,x_d) \in \RR^d \: x_{i_2}-x_{i_3}\leq 1\}$.
	Then every vertex of $\Gamma(\Cc(P), - \Cc(Q))$ belongs to $\Hc^+$, and 
	$\rho(\{p_{i_1},p_{i_2}\}),\rho(\{p_{i_2}\}),-\rho(\{q_{i_1},q_{i_3}\}),-\rho(\{q_{i_3}\})$ belong to $\Hc$.
	Since ($\rho(\{p_{i_1},p_{i_2}\})-(-\rho(\{q_{i_3}\})))=(\rho(\{p_{i_2}\})-(-\rho(\{q_{i_3}\})))-(-\rho(\{q_{i_1},q_{i_3}\})-(-\rho(\{q_{i_3}\})))$, this face is not a simplex.
	
	\textbf{((iv) $\Rightarrow$ (iii))}
	We assume that 
	$$A_2(P)=\{\{p_1,p_2\},\ldots,\{p_{2k-1},p_{2k}\},\{p_{2k+1},p_{2k+2}\},\ldots,\{p_{2k+2l-1},p_{2k+2l}\}\},$$
	$$A_2(Q)=\{\{q_1,q_2\},\ldots,\{q_{2k-1},q_{2k}\},\{q_{2k+2l+1},q_{2k+2l+2}\},\ldots,\{q_{2k+2l+2m-1},q_{2k+2l+2m}\}\},$$
	where $k,l$ and $m$ are nonnegative integers with $2k+2l+2m \leq d$.
	Then we have $\Gamma(\Cc(P), - \Cc(Q))=\text{conv}(\pm \eb_1,\ldots,\pm \eb_d,\pm (\eb_1+\eb_2),\ldots,\pm (\eb_{2k-1}+\eb_{2k}),\eb_{2k+1}+\eb_{2k+2},\ldots,\eb_{2k+2l-1}+\eb_{2k+2l},-(\eb_{2k+2l-1}+\eb_{2k+2l+2}),\ldots,-(\eb_{2k+2l+2m-1}+\eb_{2k+2l+2m}))$.
	Hence $\Gamma(\Cc(P), - \Cc(Q))$ splits into copies of the closed interval $[-1,1]$ or a del Pezzo $2$-polytope or pseudo del Pezzo $2$-polytope.
	
	\textbf{((iii) $\Rightarrow$ (ii) $\Rightarrow$(i))}
	Since $\Gamma(\Cc(P), - \Cc(Q))$ splits into copies of the closed interval $[-1,1]$ or a del Pezzo $2$-polytope or pseudo del Pezzo $2$-polytope, $\Gamma(\Cc(P), - \Cc(Q))$ is smooth.
	Moreover, in general, any smooth Fano polytope is simplicial.
\end{proof}

Next, we consider when $\Gamma(\Oc(P), - \Cc(Q))$ is smooth.
\begin{Theorem}
	\label{ocs}
	For $d \geq 2$, let $P$ and $Q$ be partially ordered sets with $|P|=|Q|=d$.
	Then the following conditions are equivalent:
	\begin{enumerate}
		\item[(i)] $\Gamma(\Oc(P), - \Cc(Q))$ is $\QQ$-factorial;
		\item[(ii)] $\Gamma(\Oc(P), - \Cc(Q))$ is smooth;
		\item[(iii)] $I(P)=\{\{p_{i_1}\},\{p_{i_1},p_{i_2}\},\ldots,\{p_{i_1},\ldots,p_{i_d}\}\}$ or \\$I(P)=\{\{p_{i_1}\},\{p_{i_2}\},\{p_{i_1},p_{i_2}\},\ldots,\{p_{i_1},\ldots,p_{i_d}\}\}$, and \\
		$A(Q)=\{\{q_{i_1}\},\{q_{i_2}\},\ldots,\{q_{i_d}\}\}$ or\\
		$A(Q)=\{\{q_{i_1}\},\{q_{i_2}\},\ldots,\{q_{i_d}\},\{q_{i_1},q_{i_2}\}\}$;
	\end{enumerate}	
\end{Theorem}
\begin{proof}
	\textbf{((i) $\Rightarrow$ (iii))}
	We may assume that $p_{i_1}$ is a minimal element of $P$ and $I'(P)=\{\{p_{i_1}\},\{p_{i_1},p_{i_2}\},\ldots,\{p_{i_1},\ldots,p_{i_d}\}\} \subset I(P)$.
	Then $x_{i_1}=1$ is a facet of  $\Oc(P)$, in particular, this is a facet of $\Gamma(\Oc(P), - \Cc(Q))$.
	Since $\Gamma(\Oc(P), - \Cc(Q))$ is simplicial, this facet is  a $(d-1)$-simplex.
	Hence there is no poset ideal $I \in I(P)$ such that $p_{i_1} \in I$ and $I \notin I'(P)$. 
	If there exists $I \in I(P)$ such that $p_{i_1} \notin I$, there exists a minimal element $p_{i_j} \in I$ of $P$.
	Then since $\{\{p_{i_1},p_{i_j}\}\}$ is a poset ideal of $P$, we have $j=2$.
	Hence we know that $I(P)=\{\{p_{i_1}\},\{p_{i_1},p_{i_2}\},\ldots,\{p_{i_1},\ldots,p_{i_d}\}\}$ or $I(P)=\{\{p_{i_1}\},\{p_{i_2}\},\{p_{i_1},p_{i_2}\},\ldots,\{p_{i_1},\ldots,p_{i_d}\}\}$.
	Also, by the proof of Theorem \ref{ccs}, we may assume that for any $J_1,J_2 \in A_2(Q)$ with $J_1 \neq J_2$, $J_1 \cap J_2=\emptyset$.
	
	We assume that $I(P)=\{\{p_{i_1}\},\{p_{i_1},p_{i_2}\},\ldots,\{p_{i_1},\ldots,p_{i_d}\}\}$.
	If $\{q_{i_j},q_{i_k}\}$ is an antichain of $Q$ with $2 \leq j < k$,
	then $x_{i_1}-x_{i_k}=1$ is a face of $\Gamma(\Oc(P), - \Cc(Q))$ and this face is not a simplex.
	Indeed, we set $\Hc_1=\{(x_1,\ldots,x_d) \in \RR^d \: x_{i_1}-x_{i_k}=1\}$ and $\Hc_1^+=\{(x_1,\ldots,x_d) \in \RR^d \: x_{i_1}-x_{i_k}\leq 1\}$.
	Then every vertex of $\Gamma(\Oc(P), - \Cc(Q))$ belongs to $\Hc_1^+$.
	Also, $\rho(\{p_{i_1}\}),\rho(\{p_{i_1},p_{i_2}\}),\ldots,\rho(\{p_{i_1},\ldots,p_{i_{k-1}}\}),-\rho(\{q_{i_k}\})$ and $-\rho(\{q_{i_j},q_{i_k}\})$ belong to $\Hc_1$.
	Since $(-\rho(\{q_{i_k}\})-\rho(\{p_{i_1}\}))=(-\rho(\{q_{i_j},q_{i_k}\})-\rho(\{p_{i_1}\}))$\\$+(\rho(\{p_{i_1},\ldots,p_{i_{j}}\})-\rho(\{p_{i_1}\}))-(\rho(\{p_{i_1},\ldots,p_{i_{j-1}}\})-\rho(\{p_{i_1}\}))$, this face is not a simplex.
	If $\{q_{i_1},q_{i_j}\}$ is an antichain of $Q$ with $3 \leq j$,
	then $-x_{i_1}+2x_{i_2}=1$ is a face of $\Gamma(\Oc(P), - \Cc(Q))$ and this face is not a simplex.
	Indeed, we set $\Hc_2=\{(x_1,\ldots,x_d) \in \RR^d \: -x_{i_1}+2x_{i_2}=1\}$ and $\Hc_2^+=\{(x_1,\ldots,x_d) \in \RR^d \: -x_{i_1}+2x_{i_2}\leq 1\}$.
	Then each vertex of $\Gamma(\Oc(P), - \Cc(Q))$ belongs to $\Hc_2^+$.
	 Also, 
	$\rho(\{p_{i_1},p_{i_2}\}),\ldots,\rho(\{p_{i_1},\ldots,p_{i_{d}}\})$,\\$-\rho(\{q_{i_1}\})$ and $-\rho(\{q_{i_1},q_{i_j}\})$ belong to $\Hc_2$.
	Hence since $-x_{i_1}+2x_{i_2}=1$ has $d+1$ vertices, this face is not a simplex.
	
	We assume that $I(P)=\{\{p_{i_1}\},\{p_{i_2}\},\{p_{i_1},p_{i_2}\},\ldots,\{p_{i_1},\ldots,p_{i_d}\}\}$.
	If $\{q_{i_j},q_{i_k}\}$ is an antichain of $Q$ with $2 \leq j < k$,
	then similarly, $x_{i_1}-x_{i_k}=1$ is a face of $\Gamma(\Oc(P), - \Cc(Q))$ and this face is not a simplex.
	If $\{q_{i_1},q_{i_j}\}$ is an antichain of $Q$ with $3 \leq j$,
	then $x_{i_2}-x_{i_j}=1$ is a face of $\Gamma(\Oc(P), - \Cc(Q))$ and this face is not a simplex.
	Indeed, we set $\Hc_3=\{(x_1,\ldots,x_d) \in \RR^d \: x_{i_2}-x_{i_j}=1\}$ and $\Hc_3^+=\{(x_1,\ldots,x_d) \in \RR^d \: x_{i_2}-x_{i_j}\leq 1\}$.
	Then every vertex of $\Gamma(\Oc(P), - \Cc(Q))$ belongs to $\Hc_1^+$, and 
	$\rho(\{p_{i_2}\}),\rho(\{p_{i_1},p_{i_2}\}),\ldots,\rho(\{p_{i_1},\ldots,p_{i_{j-1}}\}),-\rho(\{q_{i_j}\}),-\rho(\{q_{i_1},q_{i_j}\})$ belong to $\Hc_3$.
	Since $(\rho(\{p_{i_2}\})-\rho(\{p_{i_1},p_{i_2},p_{i_3}\}))=(\rho(\{p_{i_1},p_{i_2}\})-\rho(\{p_{i_1},p_{i_2},p_{i_3}\}))-(-\rho(\{q_{i_j}\})-\rho(\{p_{i_1},p_{i_2},p_{i_3}\}))+(-\rho(\{q_{i_1},q_{i_j}\})-\rho(\{p_{i_1},p_{i_2},p_{i_3}\}))$,
	this face is not a simplex.
	
	\textbf{((iii) $\Rightarrow$ (ii))}
	If $\Pc \subset \RR^d$ is a smooth Fano polytope of dimension $d$, $\Pc'=\text{conv}(\Pc, \eb_1+\eb_2+\cdots+\eb_{d+1},-\eb_{d+1}) \subset \RR^{d+1}$ is also smooth.
	Also, if $d=2$, then $\Gamma(\Oc(P), - \Cc(Q))$ is smooth.
	Hence for $d \geq 2$, we know that $\Gamma(\Oc(P), - \Cc(Q))$ is smooth.
	
	\textbf{((ii) $\Rightarrow$(i))}
	In general, any smooth Fano polytope is simplicial.
\end{proof}

Finally, we consider when $\Gamma(\Oc(P), - \Oc(Q))$ is smooth.
\begin{Theorem}
	\label{oos}
	For $d \geq 2$, let $P$ and $Q$ be partially ordered sets with $|P|=|Q|=d$.
	Assume that $P$ and $Q$ have a common linear extention.
	Then the following conditions are equivalent:
	\begin{enumerate}
		\item[(i)] $\Gamma(\Oc(P), - \Oc(Q))$ is $\QQ$-factorial;
		\item[(ii)] $\Gamma(\Oc(P), - \Oc(Q))$ is smooth;
		\item[(iii)] $I(P)=\{\{p_{i_1}\},\{p_{i_1},p_{i_2}\},\ldots,\{p_{i_1},\ldots,p_{i_d}\}\}$ or \\$I(P)=\{\{p_{i_1}\},\{p_{i_2}\},\{p_{i_1},p_{i_2}\},\ldots,\{p_{i_1},\ldots,p_{i_d}\}\}$, and \\
		$I(Q)=\{\{q_{i_1}\},\{q_{i_1},q_{i_2}\},\ldots,\{q_{i_1},\ldots,q_{i_d}\}\}$ or \\$I(Q)=\{\{q_{i_1}\},\{q_{i_2}\},\{q_{i_1},q_{i_2}\},\ldots,\{q_{i_1},\ldots,q_{i_d}\}\}$.
	\end{enumerate}	
\end{Theorem}
\begin{proof}
	\textbf{((i) $\Rightarrow$ (iii))}
	By the proof of Theorem \ref{ocs},
	 We have $I(P)=\{\{p_{i_1}\},\{p_{i_1},p_{i_2}\},$
	 $\ldots,	\{p_{i_1},\ldots,p_{i_d}\}\}$ or $I(P)=\{\{p_{i_1}\},\{p_{i_2}\},\{p_{i_1},p_{i_2}\},\ldots,\{p_{i_1},\ldots,p_{i_d}\}\}$.
	 Also, we have
	$I(Q)=\{\{q_{j_1}\},\{q_{j_1},q_{j_2}\},\ldots,\{q_{j_1},\ldots,q_{j_d}\}\}$ or $I(Q)=\{\{q_{j_1}\},\{q_{j_2}\},\{q_{j_1},q_{j_2}\},$
	$\ldots,\{q_{j_1},\ldots,q_{j_d}\}\}$.
	Since $P$ and $Q$ have a common linear extention, we may assume that for any $1 \leq k \leq d$, $i_k=j_k$.
	
	\textbf{((iii) $\Rightarrow$ (ii))}
	If $\Pc \subset \RR^d$ is a smooth Fano polytope of dimension $d$, $\Pc'=\text{conv}(\Pc,\pm (\eb_1+\eb_2+\cdots+\eb_{d+1})) \subset \RR^{d+1}$ is also smooth.
	Also, if $d=2$, then $\Gamma(\Oc(P), - \Oc(Q))$ is smooth.
	Hence for $d \geq 2$, we know that $\Gamma(\Oc(P), - \Oc(Q))$ is smooth.
	
	\textbf{((ii) $\Rightarrow$(i))}
	In general, any smooth Fano polytope is simplicial.
\end{proof}


\section{unimodularly equivalence and volume}

Let $\ZZ^{d \times d}$ denote the set of $d \times d$ integral matrices.
Recall that a matrix $A \in \ZZ^{d \times d}$ is {\em unimodular} if $\det (A) = \pm 1$.
Given integral convex polytopes $\Pc$ and $\Qc$ in $\RR^d$ of dimension $d$,
we say that $\Pc$ and $\Qc$ are {\em unimodularly equivalent}
if there exists a unimodular matrix $U \in \ZZ^{d \times d}$
and an integral vector $w$, such that $\Qc=f_U(\Pc)+w$,
where $f_U$ is the linear transformation in $\RR^d$ defined by $U$,
i.e., $f_U({\bf v}) = {\bf v} U$ for all ${\bf v} \in \RR^d$.
Clearly, if $\Pc$ and $\Qc$ are unimodularly equivalent, then
$i(\Pc,n) = i(\Qc,n)$.
Moreover, if $\Pc$ is Fano, then $w=0$ .

Let $P$ and $Q$ be partially ordered sets with $|P|=|Q|=d$.
We consider whether $\Gamma(\Oc(P), - \Oc(Q))$, $\Gamma(\Oc(P), - \Cc(Q))$ and $\Gamma(\Cc(P), - \Cc(Q))$ are unimodularly equivalent when these polytopes are smooth.
When $d=2$ these polytopes are clearly unimodularly equivalent.

For $d\geq 3$
let $P_1,P_2$ be partially ordered sets as follows.
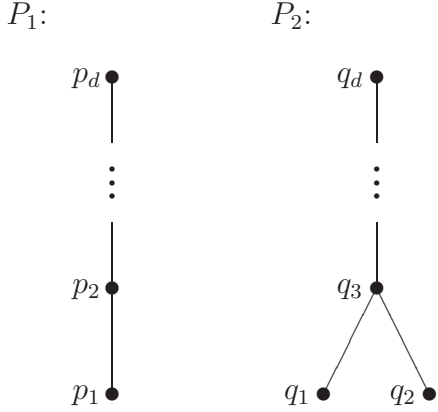
\newline
\begin{picture}(400,150)(10,50)
\put(50,170){$P_1$:}
\put(90,150){\circle*{5}}
\put(90,70){\circle*{5}}
\put(90,30){\circle*{5}}
\put(90,115){\circle*{2}}
\put(90,110){\circle*{2}}
\put(90,105){\circle*{2}}
\put(75,148){$p_d$}
\put(75,68){$p_2$}
\put(75,28){$p_{1}$}
\put(90,150){\line(0,-1){25}}
\put(90,70){\line(0,1){25}}
\put(90,30){\line(0,1){40}}

\put(150,170){$P_2$:}
\put(190,150){\circle*{5}}
\put(190,70){\circle*{5}}
\put(210,30){\circle*{5}}
\put(170,30){\circle*{5}}
\put(190,115){\circle*{2}}
\put(190,110){\circle*{2}}
\put(190,105){\circle*{2}}
\put(175,148){$q_d$}
\put(175,68){$q_3$}
\put(195,28){$q_2$}
\put(155,28){$q_1$}
\put(190,150){\line(0,-1){25}}
\put(190,70){\line(0,1){25}}
\put(210,30){\line(-1,2){20}}
\put(170,30){\line(1,2){20}}

\end{picture}\\
\\
\\
Each of $\Gamma(\Oc(P), - \Oc(Q))$, $\Gamma(\Oc(P), - \Cc(Q))$ and $\Gamma(\Cc(P), - \Cc(Q))$ is smooth if and only if $P,Q \in \{P_1,P_2\}$. 

\begin{Theorem}
	\label{equi}
	For $d \geq 3$, let $P$ and $Q $ be partially ordered sets with $|P|=|Q|=d$.
	Assume that each of $\Gamma(\Oc(P), - \Oc(Q))$, $\Gamma(\Oc(P), - \Cc(Q))$ and $\Gamma(\Cc(P), - \Cc(Q))$ is smooth.
	Then $\Gamma(\Oc(P), - \Oc(Q))$ and $\Gamma(\Cc(P), - \Cc(Q))$ are unimodularly equivalent. However, $\Gamma(\Oc(P), - \Cc(Q)$ is not unimodularly equivalent to these polytopes.
	Moreover, if $P\neq Q$, then $\Gamma(\Oc(Q), - \Cc(P)$ is also smooth and is not unimodularly equivalent to $\Gamma(\Oc(P), - \Cc(Q))$.
\end{Theorem}
\begin{proof}
	Recall each of $\Gamma(\Oc(P), - \Oc(Q))$, $\Gamma(\Oc(P), - \Cc(Q))$ and $\Gamma(\Cc(P), - \Cc(Q))$ is smooth if and only if $P,Q \in \{P_1,P_2\}$. 
	Hence we should consider the following $4$ cases.

	\textbf{(The case $P=Q=P_1$)} 
$\Gamma(\Oc(P), - \Oc(Q))$ and $\Gamma(\Cc(P), - \Cc(Q))$ are unimodularly equivalent, in particular, these polytopes are centrally symmetric.
However,
since $\Gamma(\Oc(P), - \Cc(Q))$ is not centrally symmetric, $\Gamma(\Oc(P), - \Cc(Q))$ is not unimodularly equivalent to these polytopes. 

\textbf{(The case $P=Q=P_2$)}
Similarly,
$\Gamma(\Oc(P), - \Oc(Q))$ and $\Gamma(\Cc(P), - \Cc(Q))$ are unimodularly equivalent, and $\Gamma(\Oc(P), - \Cc(Q)$ is not unimodularly equivalent to these polytopes.

\textbf{(The case $P=P_1$ and $Q=P_2$)}
$\Gamma(\Oc(P), - \Oc(Q))$ and $\Gamma(\Cc(P), - \Cc(Q))$ are unimodularly equivalent, in particular, these polytopes are pseudo-symmetric.
However, $\Gamma(\Oc(P), - \Cc(Q))$ is not unimodularly equivalent to these polytopes, since 
$|\{v \in V(\Gamma(\Oc(P), - \Cc(Q)) \: -v \in V(\Gamma(\Oc(P), - \Cc(Q))\}| \neq |\{v \in V(\Gamma(\Oc(P), - \Oc(Q)) \: -v \in V(\Gamma(\Oc(P), - \Oc(Q))\}|$, 
where we write $V(\Pc)$ for the vertex set of a polytope $\Pc$. 


\textbf{(The case $P=P_2$ and  $Q=P_1$)}
Similarly,  
$\Gamma(\Oc(P), - \Cc(Q))$ is not unimodularly equivalent to
$\Gamma(\Oc(P), - \Oc(Q))$ and $\Gamma(\Cc(P), - \Cc(Q))$.
Moreover,
$\Gamma(\Oc(P), - \Cc(Q)$ and $\Gamma(\Oc(Q), - \Cc(P)$ are not unimodularly equivalent.
Indeed, we assume that these polytopes are unimodularly equivalent.
Then there exists a unimodular matrix $U \in \ZZ^{d \times d}$
such that $\Gamma(\Oc(P), - \Cc(Q))=f_U(\Gamma(\Oc(Q), - \Cc(P))$.
Also for  $v \in \{\pm \eb_1,\pm(\eb_1+\eb_2) \}$, there exists  $v' \in \{\pm \eb_1,\pm \eb_2\}$ such that $f(v)=v'$.

If $f_U(\eb_1)=\eb_1$ and $f_U(\eb_1+\eb_2)=\eb_2$, we have 
\begin{displaymath}
U=\left(  
\begin{array}{ccccc}
1 & 0   & 0 & \cdots   & 0 \\
-1  & 1 & 0 &  \cdots  & 0 \\
u_{31}  & u_{32} & u_{33} &  \cdots  & u_{3d} \\
\vdots & \vdots & \vdots  &  \ddots   & \vdots  \\
u_{d1}  & u_{d2} & u_{d3} &  \cdots  & u_{dd} \\
\end{array}
\right)
\in \ZZ^{d \times d}.
\end{displaymath}
Then $f(-\eb_2)=\eb_1-\eb_2 \notin V(\Gamma(\Oc(P), - \Cc(Q)))$.

If $f_U(\eb_1)=\eb_1$ and $f_U(\eb_1+\eb_2)=-\eb_2$, we have 
\begin{displaymath}
U=\left(  
\begin{array}{ccccc}
1 & 0   & 0 & \cdots   & 0 \\
-1  & -1 & 0 &  \cdots  & 0 \\
u_{31}  & u_{32} & u_{33} &  \cdots  & u_{3d} \\
\vdots & \vdots & \vdots  &  \ddots   & \vdots  \\
u_{d1}  & u_{d2} & u_{d3} &  \cdots  & u_{dd} \\
\end{array}
\right)
\in \ZZ^{d \times d}.
\end{displaymath}
Then  $f(\eb_1+\eb_2+\eb_3)=(u_{31},u_{32}-1,u_{33},\ldots,u_{3d})$ and $f(-\eb_3)=(-u_{31},\ldots,-u_{3d})$.
Since $\Gamma(\Oc(P), - \Cc(Q)$ is a ($-1,0,1$)-polytope, $u_{32}=0$ or $1$.
Then $f(\eb_1+\eb_2+\eb_3) = -\eb_2$ or $f(-\eb_3) = -\eb_2 $, contradiction.

If $f_U(\eb_1)=-\eb_1$ and $f_U(\eb_1+\eb_2)=\eb_2$, we have 
\begin{displaymath}
U=\left(  
\begin{array}{ccccc}
-1 & 0   & 0 & \cdots   & 0 \\
1  & 1 & 0 &  \cdots  & 0 \\
u_{31}  & u_{32} & u_{33} &  \cdots  & u_{3d} \\
\vdots & \vdots & \vdots  &  \ddots   & \vdots  \\
u_{d1}  & u_{d2} & u_{d3} &  \cdots  & u_{dd} \\
\end{array}
\right)
\in \ZZ^{d \times d}.
\end{displaymath}
Then  $f(\eb_1+\eb_2+\eb_3)=(u_{31},u_{32}+1,u_{33},\ldots,u_{3d})$ and $f(-\eb_3)=(-u_{31},\ldots,-u_{3d})$.
Since $\Gamma(\Oc(P), - \Cc(Q))$ is a ($-1,0,1$)-polytope, $u_{32}=0$ or $-1$.
Then $f(\eb_1+\eb_2+\eb_3) = \eb_2$ or $f(-\eb_3) = \eb_2 $, contradiction.

If $f_U(\eb_1)=-\eb_1$ and $f_U(\eb_1+\eb_2)=-\eb_2$, we have 
\begin{displaymath}
U=\left(  
\begin{array}{ccccc}
-1 & 0   & 0 & \cdots   & 0 \\
1  & -1 & 0 &  \cdots  & 0 \\
u_{31}  & u_{32} & u_{33} &  \cdots  & u_{3d} \\
\vdots & \vdots & \vdots  &  \ddots   & \vdots  \\
u_{d1}  & u_{d2} & u_{d3} &  \cdots  & u_{dd} \\
\end{array}
\right)
\in \ZZ^{d \times d}.
\end{displaymath}
Then $f(-\eb_2)=-\eb_1+\eb_2 \notin V(\Gamma(\Oc(P), - \Cc(Q)))$.

Therefore,  $\Gamma(\Oc(P), - \Cc(Q)$ and $\Gamma(\Oc(Q), - \Cc(P)$ are not unimodularly equivalent.
\end{proof}

By Theorems \ref{Ehrhart} and \ref{equi},
the following corollary immidiately follows.
\begin{Corollary}
	For any $d \geq 3$, there exist smooth Fano polytopes $\Pc$ and
	$\Qc$ such that the following conditions satisfied:
	\begin{itemize}
		\item $\Pc$ and $\Qc$ have same Ehrhart polynomial.
		\item $\Pc$ and $\Qc$ are not unimodularly equivalent.
	\end{itemize}
\end{Corollary}

Let $\Pc$ be an integral convex polytope of dimension $d$.
We write  $\Vol(\Pc)$ for the \textit{normalized volume} of $\Pc$;
it is equal to $d!$ times the usual Euclidean volume.
It is known that
if $\Pc_1$ is a $d_1$-dimensional Gorenstein Fano polytope in $\RR^{d_1}$ 
and $\Pc_2$ is a $d_2$-dimensional integral convex polytope in $\RR^{d_2}$
with $\bf{0} \in \Pc_1\setminus \partial\Pc_1$, then
$$\text{Vol}(\Pc_1\oplus \Pc_2)=\text{Vol}(\Pc_1)\text{Vol}(\Pc_2)$$
(see \cite{Braun}).
We let $l,m,n$ be nonnegative integers and 
$$\Pc=(\oplus_{l} L) \oplus (\oplus_m \tilde{V}_2) \oplus (\oplus_n V_2),$$
where $L$ is the closed interval $[-1,1]$.
Since $\Vol(L)=2,\Vol(\tilde{V}_2)=5$ and $\Vol(V_2)=6$, we have $\Vol(\Pc)=2^l \cdot 5^m \cdot 6^n.$

Finally, we consider the volume of each of $\Gamma(\Oc(P), - \Oc(Q))$, $\Gamma(\Oc(P), - \Cc(Q))$ and $\Gamma(\Cc(P), - \Cc(Q))$ when these polytopes are smooth.

\begin{Example}{\em
(i)	Let $P=Q=P_1$.
Then $\Gamma(\Cc(P), - \Cc(Q)$ is unimodularly equivalent to
$\oplus_d L$.
Hence 
we know the normalized volume of each of $\Gamma(\Oc(P), - \Oc(Q))$, $\Gamma(\Oc(P), - \Cc(Q))$ and $\Gamma(\Cc(P), - \Cc(Q))$ is equal to $2^d$ by Theorem \ref{Ehrhart}.

(ii)Let $P=Q=P_2$.
Then $\Gamma(\Cc(P), - \Cc(Q)$ is unimodularly equivalent to
$(\oplus_{d-2} L)\oplus V_2$.
Hence 
the normalized volume of each of $\Gamma(\Oc(P), - \Oc(Q))$, $\Gamma(\Oc(P), - \Cc(Q))$ and $\Gamma(\Cc(P), - \Cc(Q))$ is equal to $2^{d-2}\cdot 6$.

(iii) Let $P=P_1$ and $Q=P_2$.
Then $\Gamma(\Cc(P), - \Cc(Q)$ is unimodularly equivalent to
$(\oplus_{d-2} L)\oplus \tilde{V}_2$.
Hence 
the normalized volume of each of $\Gamma(\Oc(P), - \Oc(Q))$, $\Gamma(\Oc(P), - \Cc(Q))$ and $\Gamma(\Cc(P), - \Cc(Q))$ is equal to $2^{d-2}\cdot 5$.
In particular, the normalized volume of $\Gamma(\Oc(Q), - \Oc(P))$ is also $2^{d-2}\cdot 5$.
}
\end{Example}

\end{document}